\documentclass[11pt,a4paper]{article}

\usepackage[utf8]{inputenc}
\usepackage{authblk}
\author[1]{Lea F\"ocke}
\author[2,3]{Daniel Baumgarten}
\author[1]{Martin Burger}
\affil[1]{Institute for Analysis and Numerics, Westfälische Wilhelms-Universität Münster, Germany,}
\affil[2]{Institute of Electrical and Biomedical Engineering, Private University of Health Sciences, Medical Informatics and Technology, Hall in Tirol, Austria,}
\affil[3]{Institute of Biomedical Engineering and Informatics, Technische Universität Ilmenau, Germany,}
\title{The Inverse Problem of Magnetorelaxometry Imaging}
\date{March 19, 2018}

\usepackage[utf8]{inputenc}
\usepackage{authblk}

\usepackage[english]{babel}
\usepackage{lipsum}
\usepackage{microtype}
\usepackage{booktabs} % Horizontal rules in tables
\usepackage{float} % Required for tables and figures in the multi-column environment - they need to be placed in specific locations with the [H] (e.g. \begin{table}[H])
\usepackage{paralist} % Used for the compactitem environment which makes bullet points with less space between them

\usepackage{amsmath}

\usepackage{abstract} % Allows abstract customization
\providecommand{\keywords}[1]{\textbf{\textit{keywords---}} #1}
\usepackage{marvosym}
\usepackage{setspace}
\usepackage{scrhack}
\usepackage{array}
\usepackage{graphicx}
\usepackage{caption}
\usepackage{subcaption}
\usepackage{amsmath}
\usepackage{amssymb}
\usepackage{empheq}
\usepackage{dsfont}		
\usepackage{amstext}
\usepackage{amsfonts}
\usepackage{amsthm}
\usepackage{shadethm}
\usepackage{wasysym}
\usepackage{xcolor}
\usepackage{listings}
\usepackage{algorithm}
\usepackage{algpseudocode}
\usepackage{parskip}
\usepackage{ifthen}
\usepackage{contour}
\usepackage{todonotes}
\usepackage{accents}
\usepackage{multirow,tabularx}
\newcolumntype{B}[1]{>{\centering\arraybackslash}m{#1}}

\usepackage{ulem}
\usepackage{nomencl}
\usepackage{capt-of}
\usepackage{afterpage}
\usepackage{ltxtable}  
\usepackage{enumitem}
\usepackage{stmaryrd}
\usepackage{csquotes}

\usetikzlibrary{intersections}

\usepackage{pgfplots} \pgfplotsset{compat=1.5, width=\textwidth, height=4cm}
\usepackage{tikz}
\usepackage[mode=buildnew]{standalone}
\DeclareMathOperator*{\argmin}{\arg \min}%
\newcommand{\mymin}[1]{\underset{#1}{\min}\ }%
\newcommand{\myargmin}[1]{\underset{#1}{\argmin}\ }%

\newcommand{\R}{\mathbb{R}}

\newcommand{\defgr}{\mathrel{\mathop:\!\!=}}

\newcommand\cellwidth{\TX@col@width}
\newtheoremstyle{mytheoremstyle}
{1.5em} % Space above
{1em} % Space below
{} % Body font
{} % Indent amount
{\bfseries} % Theorem head font
{\\} % Punctuation after theorem head
{.5em} % Space after theorem head
{\thmname{#1}\thmnumber{ #2}: \thmnote{\normalfont\textsc{#3}}} % Theorem head spec (can be left empty, meaning `normal')

\newtheoremstyle{myremstyle}
{1.5em} % Space above
{1em} % Space below
{\normalfont} % Body font
{} % Indent amount
{\bfseries} % Theorem head font
{} % Punctuation after theorem head
{.5em} % Space after theorem head
{\thmname{#1}\thmnumber{ #2}: \thmnote{\normalfont\textsc{#3}}} % Theorem head spec (can be left empty, meaning `normal')

\newtheoremstyle{myproofstyle}
{0.5em} % Space above
{1em} % Space below
{} % Body font
{} % Indent amount
{\bfseries} % Theorem head font
{} % Punctuation after theorem head
{.5em} % Space after theorem head
{\thmname{\textsc{#1}}\thmnote{\textmd{ of} #3}:} % Theorem head spec (can be left empty, meaning `normal')

\theoremstyle{mytheoremstyle}
\newtheorem{lem}{Lemma}[section]
\theoremstyle{myremstyle}
\theoremstyle{mytheoremstyle}
\newtheorem{thm}[lem]{Theorem}
\theoremstyle{mytheoremstyle}
\newshadetheorem{defiLem}[lem]{Definition + Lemma}
\theoremstyle{mytheoremstyle}
\newshadetheorem{defiThem}[lem]{Definition + Theorem}
\theoremstyle{mytheoremstyle}
\newshadetheorem{defi}[lem]{Definition}
\theoremstyle{mytheoremstyle}
\newshadetheorem{cor}[lem]{Corollary}
\theoremstyle{mytheoremstyle}

\theoremstyle{myproofstyle}

\usepackage{a4wide}
\begin{document}

\maketitle

\begin{abstract}
The aim of this paper is to provide a solid mathematical discussion of the inverse problem in Magnetorelaxometry Imaging (MRXI), a currently developed technique for quantitative biomedical imaging using magnetic nanoparticles. 
We provide a detailed discussion of the mathematical modeling of the forward problems including possible ways to activate and measure, leading to a severely ill-posed linear inverse problem.
Moreover, we formulate an idealized version of the inverse problem for infinitesimal small activation coils, which allows for a more detailed  analysis of uniqueness issues.

We propose a variational regularization approach to compute stable approximations of the solution and discuss its discretization and numerical solution.
Results on synthetic are presented and improvements to methods used previously in practice are demonstrated.
Finally we give an outlook to further questions and in particular experimental design.
	%\noindent
	%\input{chapters/abstract}
\end{abstract}

\section{Introduction}
Measuring and analyzing magnetic nanoparticles (MNP) for medical applications is currently under heavy research.
For example, MNP are employed for novel cancer therapy techniques referred to as Magnetic Hyperthermia \cite{Hiergeist1999} and Magnetic Drug Targeting \cite{Alexiou2011}.
For both applications, the amount and distribution of the magnetic nanoparticles in the tissue are crucial for efficacy and safety of the therapy.
Magnetorelaxometry (MRX) is able to determine the amount of magnetic nanoparticles based on their magnetic response to an external magnetic field \cite{Wiekhorst2012}.
On this basis Magnetorelaxometry Imaging (MRXI) has been developed as a novel imaging modality aiming at the acquisition of three dimensional and quantitative reconstructions of the particle distribution \cite{Liebl2015}.
This knowledge is crucial for monitoring the mentioned therapies and can further be used to validate assumptions about the distribution, finally leading to a more precise and safe treatment.

MRXI can be characterized by two alternating phases \cite{Baumgarten2008}: first, the magnetic moment of the MNP in the area of interest is aligned by coil induced magnetic fields.
Second, after these coils are switched off the MNP show a magnetic relaxation that is measured with highly sensitive sensors outside the tissue.
These steps can be executed multiple times with different excitation fields \cite{Liebl2014}. The problem of determining the quantitative distribution of the MNP from these measurements can be formulated as an inverse problem \cite{Baumgarten2008}, which we will elaborate further in Section \ref{sec:forwardmodel} and \ref{sec:inverseproblem}.

Based on this general approach, several studies have been published concerning activation patterns and coil positioning for inhomogeneous excitation fields \cite{Baumgarten2010, Crevecoeur2012, Coene2012, Baumgarten2014, Coene2015}.
The main interest here has been to improve the system condition to gain reconstruction quality using basic regularization techniques, i.e. least squares solution using the pseudo inverse, Truncated Singular Value Decomposition (TSVD) and Tikhonov regularization.
On the other hand, nonlinear regularization techniques have shown promising results for image reconstruction in undersampled MRI, CT and PET cf. e.g. \cite{Sawatzky2011}) in particular using the total variation  as part of the variational model.

In this paper, we will recall he model originally presented by Baumgarten et al. and Liebl et al. \cite{Baumgarten2008, Liebl2014} and put it into a mathematically rigorous inverse problems framework.
On this basis we will determine the inverse problem of MRXI and investigate ill-posedness and uniqueness issues.
Motivated by these insights we will present a variational model to find a meaningful solution to the image reconstruction problem, where we apply the Total Variation (TV) regularization in combination with a positivity constraint leading to
$$\mymin{c}\Vert Kc-g\Vert_2^2 + \alpha\operatorname{TV}(c) + \chi_{\geq 0}(c).$$
In the end we will provide a simulation setup and compare results using nonlinear regularization techniques with previously u used techniques for MRXI. Moreover, we will provide a preliminary discussion of the impact of different activation strategies, a crucial issue for future research.

\section{The Forward Model}\label{sec:forwardmodel}
In this section we describe the basic principles of MRXI, how data is acquired and processed.
We roll out the general idea in Subsection \ref{ssec:principles} and describe a mathematical model for a 3D environment in Subsection \ref{ssec:mathmodel}. In the end we provide a mathematically idealized model in Subsection \ref{ssec:idealmodel}.

\subsection{Basic Idea and Physical Principles}
\label{ssec:principles}
%Magnetorelaxometry Imaging (MRXI) uses the magnetic response of aligned magnetic nanoparticles (MNP).
Magnetic nanoparticles for biomedical applications usually consist of a magnetic iron oxide core with a diameter of a few up to about 30 nanometers surrounded by a non-magnetic shell layer. 
The magnetic core of these particles usually contains a single magnetic domain and can therefore be modeled as a magnetic dipole, thus having an magnetization depending on core size and material used. This magnetization and therewith the magnetic moment can be oriented by external magnetic fields either within the particle's core in term of Néel motion \cite{Neel1949} or by rotation of the whole particle in terms of Brownian motion \cite{Brown1963}.

%The magnetic moment of particles of this size is very small. Therefore, the magnetic field of a single particle is in general too weak to be measured.
%In order to acquire a measurable response, the fact is used that magnetic fields are additive and provide a overlaid and more distinctive field in superposition.
%That means that the resulting field of overlaid induced fields can gain or loose intensity in relation to the particle alignment.

Following the idea of aligning the particles and afterwards measuring the response, we obtain two distinct phases that are implemented in MRXI: First, during the so called 'Excitation Phase' the particles are exposed to a magnetic field strong enough to reorientate the magnetization of the particles. At this point, the fields of the aligned particles add up to a measurable superposition field.
For the second 'Relaxation Phase', the external coils aligning the particles are switched off.
Due to several reasons, the particle's dipole orientation shifts in arbitrary direction yielding a relaxation signal that is measured with highly sensitive sensors. Currently, usually SQUIDs (Superconducting QUantum Interference Device \cite{Liebl2012}) are employed.
The combination of these two previously described distinct phases is called to be one iteration of MRXI. 
We provide an intuitive illustration of the entire cycle in Figure \ref{img:phases}.
Since the excitation fields are a few orders of magnitudes larger compared to the fields induced by the aligned particle, a delay time between the phases is required to ensure that the excitation fields do not influence the data acquisition.
We will give more detailed information on the referred time steps $t_0, \ldots, t_5$ in the subsequent section. 

\begin{figure}
	\centering
	\includegraphics[width=0.75\textwidth]{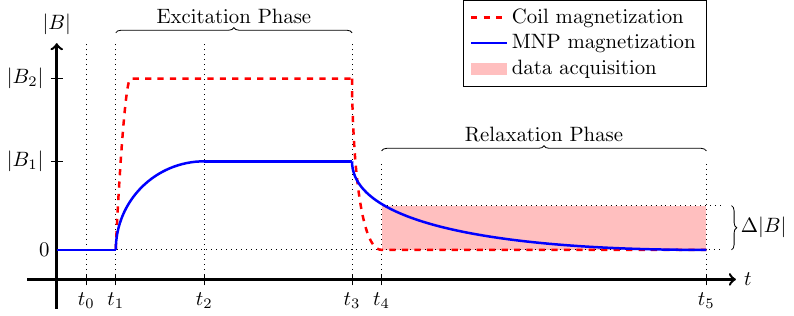}%  
	%\includestandalone[width=0.75\textwidth]{images/phases}
	\caption{\footnotesize Illustration of the absolute strength of the magnetic field induced by external coils and magnetic nanoparticles during MRXI. Timestep $t_0$ describes the default state of the system without any fields applied.
		In $t_1$ the external magnetic field is activated and induces a magnetic field with strength $\vert B_2\vert$.
		The particle alignment reaches a maximum at $t_2$ with an induced magnetic field $\vert B_1\vert$.
		In $t_3$ the coils are disabled and are fully deactivated in $t_4$.
		The data acquisition of the particles' induced fields is carried out during the time interval $\left[ t_4, t_5\right]$ (compare pink area).}
	\label{img:phases}
\end{figure}

\subsection{Mathematical Model}\label{sec:mathmodel}
\label{ssec:mathmodel}

Let us now detail the mathematical model for MRX Imaging.
For this sake we denote by $\Omega\subset\R^3$ a bounded domain describing the region of interest holding magnetic nanoparticles.
Then excitation coils and measuring sensors are positioned in $\R^3\setminus\Omega$.
In general coil and sensors may be in the same position in $\R^3\setminus\Omega$ since both are distinct processes and, in theory, can be exchanged during an MRXI iteration.

We define $c\in\mathbf{L}^2(\Omega)$ as a nonnegative density function of magnetic nanoparticles with compact support on the domain $\Omega$, it will be the unknown to be determined in the inverse problem.
We have already stated that the particle properties depend on multiple factors, including core size and material, however for our purposes we assume a constant particle base magnetization $\mathbf{m}_0$.

Then we define a vector field $\mathbf{m}\in\mathcal{L}^2(\Omega, \R^3)$ defining the magnetic moment of a corresponding particle density $c$.
We assume that in the initial state $(t_0)$ and in the full relaxation state $(t_5)$ all particles are orientated randomly on a microscopic scale.
As a result the particle's magnetization cancels out on a macroscopic scale and therefore we demand $\mathbf{m}_{t_0}=\mathbf{m}_{t_5}=0$.

In the following we specify an excitation coil $\alpha\in\mathcal{A}$, where $\mathcal{A}$ defines the set of all coils. $\alpha = (\varphi_\alpha, I_\alpha)$ consist of a conductor path $\varphi_\alpha$ and an activation current $I_\alpha$.
However, for our purposes we only consider coil activations with a constant current $I_\alpha\equiv1$.
The coil conductor is defined as a curve in $\R^3\setminus\Omega$ namely $\varphi_\alpha\in\mathcal{C}^2([0,L_\alpha],\R^3\setminus\Omega)$, where $L_\alpha$ is the length of the curve that is assumed to be given in arclength parametrization.
Then the Biot-Savart-Law (cf. \cite{Jackson1999}) provides a connection between the coil conductor path $\varphi_\alpha$ and the resulting magnetic field $\mathbf{B}$ in $w\in\Omega$:
%\begin{align}
\begin{equation}
	\mathbf{B}_\alpha^\textbf{coil}\colon\Omega \rightarrow\R^3, \quad
	w \mapsto\vartheta\int\limits_0^{L_\alpha}\varphi_\alpha^\prime(s)\times\left(\frac{w-\varphi_\alpha(s)}{\left\vert w-\varphi_\alpha(s)\right\vert^3}\right)ds \label{eq:biotsavart}
\end{equation}
%\end{align}

Note that $w-\varphi_\alpha(s)$ with $w\in\Omega$ and $\varphi_\alpha(s)\in\R^3\setminus\Omega$ is well defined as an element of $\R^3$.
Here $\vartheta$ provides a collection of physical constants, including activation current $I_\alpha$.

Given an external magnetic field $\mathbf{B}$, i.e. $\mathbf{B}_\alpha^{\mathbf{coil}}$ as in Equation \eqref{eq:biotsavart}, the behavior of magnetic particles with a magnetization $\mathbf{m}$ can be described by the well-known Langevin function $\mathbf{L}\colon x\mapsto\coth(x)-\frac{1}{x}$ illustrated in Figure \ref{img:langevin}.
For a given magnetic field $\mathbf{B}$ and a particle density $c$ the resulting magnetization of the particles in $w\in\Omega$ is given by 
\begin{equation}
%\begin{align}	
	\mathbf{m}\colon \Omega\rightarrow\R^3, \quad
	w \mapsto\mathbf{L}(\mathbf{B}(w)) c(w)\label{eq:langevinfunctional}
%\end{align}	
\end{equation}

\begin{figure}
	\centering
		\includegraphics[width=0.75\textwidth]{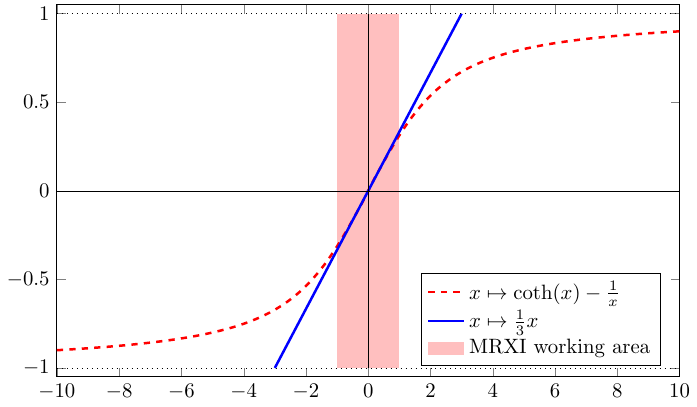}
		%\includestandalone[width=0.75\textwidth]{images/langevin}
	\caption{\footnotesize Langevin Function, linearization and working area of MRXI: The blue curve is the Langevin function defined as $\mathbf{L}\colon x\mapsto \coth(x)-\frac{1}{x}$, the red line its linearization $\mathbf{L}$ in $x=0$: $x\mapsto\frac{1}{3}x$. The pink area illustrates the approximate value range for the Langevin functional that is used for the forward operator of MRXI.}
	\label{img:langevin}
\end{figure}%
%\begin{align*}	
%	\mathbf{m}\colon\Omega&\rightarrow\R^3\\
%	w&\mapsto\vartheta\mathcal{L}\left(\vartheta^\prime \mathbf{B}(w)\right)c(w)\\
%\end{align*}
%Here $c(w)$ describes the corresponding particle density in $\Omega$

For weak magnetic fields the magnetization operates in the linear range of the Langevin function and can be approximated well by $\mathcal{L}(x)\approx\frac{x}{3}$.
Note that the correct shape and scale of the Langevin function depends on particle size, particle base magnetization, temperature, particle domain and other physical constants.
Here we set all these variables as suitable for the application of the mentioned linearization.
Then the particle magnetization $\mathbf{m}$ after being exposed to a magnetic field $\mathbf{B}_\alpha$ is given by
\begin{equation}
%\begin{align}
	\mathbf{m}_\alpha\colon \Omega\rightarrow\R^3, \quad
	w \mapsto\frac{1}3 \mathbf{B}_\alpha^\textbf{coil}(w)c(w).\label{eq:dipolemagnetization}
%\end{align}	
\end{equation}

For the sensors, which measure the magnetic field differences over time, we define a measuring space $\Sigma=(\R^3\setminus\Omega)\times\mathcal{S}^2$.
Moreover we define $\sigma = (\sigma_x, \sigma_n)\in\Sigma$ as a single sensor with position $\sigma_x\in\R^3\setminus\Omega$ and orientation $\sigma_n\in \mathcal{S}^2\subseteq\R^3$, that only acquires magnetic fields in direction of its given orientation $\sigma_n$.
For a given magnetization peak $\mathbf{m} \delta_w$ in $w\in\Omega$ the sensor in $\sigma_x$ yields a magnetic field measured in direction $\sigma_n$: 
\begin{align}
	\mathbf{B}^\textbf{meas}\colon\Omega\times\Sigma&\rightarrow\R\nonumber\\
	\left(w,\sigma\right)&\mapsto\sigma_n\cdot\left(\left(\frac{3\left(\sigma_x-w\right)\otimes\left(\sigma_x-w\right)}{\left\vert\sigma_x-w\right\vert^5}-\frac{\mathbb{I}}{\left\vert\sigma_x-w\right\vert^3}\right)\mathbf{m}(w)\right)\label{eq:measmagnetization}
\end{align}

where $\mathbb{I}\in\R^{3\times3}$ denotes the identity matrix. 

With the modeling of each part at hand we can, in a next step, assemble the forward operator.
Again we refer to the time steps as seen in Figure \ref{img:phases}.
As stated before  in the initial state ($t_0$) we have no magnetization in $\Omega$, therefore $\mathbf{m}=0$.
Then at time step $t_1$ coil $\alpha$ is activated and the resulting magnetic field $\mathbf{B}_\alpha^{\textbf{coil}}$ is given as of Equation \eqref{eq:biotsavart}.
The applied fields initiate the reorientation process.
The stable state is described by Equation \eqref{eq:dipolemagnetization} resulting in a magnetization $\mathbf{m}_\alpha$.
The particles' reorientation process reaches a stable state in time step $t_2$.
Shortly after this process is done the coils are deactivated in $t_3$.
In practice, it takes a short time interval to fully disable the excitation coils, therefore the data acquisition is started at time step $t_4$ after full disappearance of $\mathbf{B}^\textbf{coil}$.
Note that in practice magnetic field sensors can only register changing fields over time.
Therefore the acquired data represents the change in the magnetic field $\Delta\mathbf{B}^\textbf{meas} = \mathbf{B}^\textbf{meas}_{t_4} - \mathbf{B}^\textbf{meas}_{t_5}$.
For this model we demand the following:
first the particle reach a full relaxed state in $t_5$, therefore $\mathbf{B}^\textbf{meas}_{t_5} = 0$, and second the switch off interval for the excitation coil can be ignored (i.e. $t_3=t_4$).
As a result we have $\Delta\mathbf{B}^\textbf{meas} = \mathbf{B}^\textbf{meas}_{t_3}$ as of Equation \eqref{eq:measmagnetization}.

In short hand notation the measured magnetic field $\mathbf{B}^\textbf{meas}_\alpha(\sigma)$ in $\sigma$ induced by coil $\alpha$ is

\begin{equation}
	\mathbf{B}^\textbf{meas}_\alpha(\sigma) = \mathbf{k}_\alpha(w,\sigma)c(w)\label{eq:modelwithkernel}
\end{equation}

%With $\mathbbm{r}\defgr\sigma_x-w$ can denote the acquired measurement for coil $\alpha$ activation:
%\begin{align}
%\mathbf{B}_\alpha^\textbf{meas}\colon\left(w,\sigma\right)&\mapsto\phantom=\sigma_n\cdot\left(\left(\frac{\mathbbm{r}\otimes\mathbbm{r}}{\left\vert\mathbbm{r}\right\vert^5}-\frac{\mathbb{1}}{\left\vert\mathbbm{r}\right\vert^3}\right)\mathbf{m}_\alpha(w)\right)\nonumber\\
%&\phantom\mapsto=\mathbf{k}_\alpha(w,\sigma)c(w)\label{eq:modelwithkernel}
%\end{align}

where we define the kernel of the measurement process as

\begin{equation}\label{eq:mrxikernel}
	\mathbf{k}_\alpha(w,\sigma)= \sigma_n\cdot\left(\left(\frac{3(\sigma_x-w)\otimes(\sigma_x-w)}{\left\vert \sigma_x-w\right\vert^5}-\frac{\mathbb{I}}{\left\vert \sigma_x-w\right\vert^3}\right)\mathbf{B}_\alpha^\textbf{coil}(w)\right).
\end{equation} 

Now we define the forward Operator $\mathbf{K}_\alpha$ for a single given coil activation $\alpha$.
Here, the sensor measures the combined response of all particles, hence using \eqref{eq:modelwithkernel} we have
\begin{equation}
%\begin{align}
\mathbf{K}_\alpha\colon\mathcal{L}^2(\Omega) \rightarrow\mathcal{L}^2(\Sigma), \quad
c \mapsto\left[\sigma\mapsto\int\limits_\Omega\mathbf{k}_\alpha(w,\sigma)c(w)d^3w\right]\label{eq:fredholmformulation}
%\end{align}	
\end{equation}

We finally mention that obviously in reality the magnetic fields of all particles directly combine and produce the directional measurement of the sensors together.
In our derivation this means that we first need the convolutional integral and then evaluate at $\sigma_x$ and take the scalar product with $\sigma_n$. However, due to the linearity of the operations this procedure is equivalent to our derivation leading directly to the kernel function ${\bf k}_\alpha$.

Mathematically $\mathbf{K}_\alpha$ maps a particle density $c$ to elements of the measurement space $\mathcal{L}^2(\Sigma)$. Note that the support of the activating coils as well as the sensors are located outside $\Omega$, hence $k_\alpha$ is a bounded integral kernel. Thus, the well-definedness and boundedness of $\mathbf{K}_\alpha$ in the above spaces follows from known results on Fredholm integral operators (cf. \cite{Engl2013}). The full forward operator $\mathbf{K}$ is then the collection of all $\mathbf{K}_\alpha$ and the inverse problem consists in (approximatively) inverting $\mathbf{K}$.

So far the set $\mathcal{A}$ parameterizing the activations was rather general, hence we discuss the possible strategies for activating the coils, which determine $\mathcal{A}$.
Note that this is also important for a sound mathematical formulation, since in a general formulation of activations solely parametrized by an index set $\mathcal{A}$ it is not even clear in which spaces to define measurements respectively how to set up the operator $\mathbf{K}$. The first option is to simply take $\mathcal{A}$ as a finite set, corresponding to the practical realization of a measurement. However, the shape and even size for this set may change in different experiments, since one is rather free about how and where to place the coils. In order to fully exploit the capabilities of MRX Imaging it seems more reasonable to construct a continuous model and interpret the practical model as a sampling thereof. For this sake we notice that there are mainly three options for varying the coils: the type of the coil (shape of the curve $\varphi_\alpha$), its position in space (corresponding to the center of mass $y \in (\R^3 \setminus \Omega)$ of $\varphi_\alpha$) and the orientation $\eta \in \mathcal{S}^2$ (corresponding to a rotation of the coil). It seems reasonable to assume that there is only a finite number $M$ of different coil types and that activations are not carried out at arbitrary high distance from $\Omega$. The possible sensors are collected in a bounded subset of $\Sigma=(\R^3\setminus\Omega) \times \mathcal{S}^2$ already. Hence we can encode the actual measurements into a probability measure $\mu$ with compact support on 
$$ \mathcal{M}:=\underbrace{\{1,\ldots,M\} \times \Sigma}_{\text{coils}} \times \underbrace{\Sigma}_{\text{sensors}},$$ 
where we assume that $\mu$ is a product measure of the form $ \mu = \mu^{coil} \otimes \mu^{meas} $, where $\mu^{coil}$ is a probability measure on  $\{1,\ldots,M\} \times \Sigma$ and $\mu^{meas}$ a probability measure on $\Sigma$.
Then we can define the forward operator 
$$ \mathbf{K} \colon\mathcal{L}^2(\Omega)\rightarrow\mathcal{L}^2(\mathcal{M}; \mu ). $$ 
The case of a finite number of measurements is then a specific realization that we obtain by choosing $\mu^{coil}$ respectively $\mu^{meas}$ as concentrated measure.
Noticing that 
\begin{align*}
  \Vert \mathbf{K} c \Vert_{\mathcal{L}^2 ( \mathcal{M}; \mu )}^2 &= 
\int \limits_{\{1,\ldots,M\} \times \Sigma} \Vert \mathbf{K}_\alpha c \Vert_{\mathcal{L}^2
( \Sigma; \mu^{meas} )}^2 ~d\mu^{coil}(\alpha)\\
&\leq \sup_{\alpha} \Vert \mathbf{K}_\alpha \Vert_{\mathcal{L}^2
( \Sigma; \mu^{meas} )}^2,
\end{align*}
%\todoi{$\Vert \mathbf{K}_\alpha c \Vert_{\mathcal{L}^2	( \Sigma; \mu^{meas} )}^2$?}
boundedness of $\mathbf{K}$ follows from the uniform boundedness of the operators $\mathbf{K}_\alpha$, which is a straightforward estimate if all possible coil locations are outside $\Omega$. Indeed, by analogous arguments we can even show that the extension 
\begin{align*}
\Vert \mathbf{K} c \Vert_{\mathcal{L}^1 ( \mathcal{M}; \mu )}^2 &= 
\int \limits_{\{1,\ldots,M\} \times \Sigma} \Vert \mathbf{K}_\alpha c \Vert_{\mathcal{L}^1
( \Sigma; \mu^{meas} )}^2 ~d\mu^{coil}(\alpha)\\
&\leq \sup_{\alpha}\Vert \mathbf{K}_\alpha \Vert_{\mathcal{L}^1
( \Sigma; \mu^{meas} )}^2
\end{align*}
is well-defined and bounded.
 
In practice one would sometimes like to use multiple coil activations within the set of coils $\mathcal{A}$ simultaneously, also allowing different coil currents to vary the resulting field strength. Let $\mathcal{B}$ the set of all possible activation patterns. Then we can introduce weighting parameters $\omega_\alpha^\beta \in\R_{\geq0}$, where $\beta\in\mathcal{B}$ defines a specific pattern, and write
\begin{equation} \label{eq:multipleactivation}
\mathbf{\tilde K} = \left(\sum\limits_{\alpha\in\mathcal{A}}\omega_\alpha^\beta \mathbf{K}_\alpha \right)_{\beta \in \mathcal{B}}.
\end{equation}
As a result the choice of the weights $\omega_\alpha^\beta$ defines a specific activation pattern. Note that even if $\mathcal{A}$ is not a finite set, we will only have a finite number of nonzero weights in any measurement, the above sums have thus to be interpreted as finite ones and no convergence issues arise. The case of directly measuring $\mathbf{K}_\alpha$ is of course a special case where one weight is equal to one and the others are vanishing.

\subsection{An Idealized Mathematical Model}
\label{ssec:idealmodel}
In order to gain further understanding of the capabilities and mathematical structure of MRXI, it will be useful to study an idealized model that does not need to take into account the fine details of the coil. We approximate the activation by the limit of a small coil, with $y_\alpha$ the center of mass of $\varphi_\alpha$. Then we can approximate
\begin{align*}
	\int\limits_0^{L_\alpha}\varphi_\alpha^\prime(s)\times\left(\frac{w-\varphi_\alpha(s)}{\left\vert w-\varphi_\alpha(s)\right\vert^3}\right)ds &\approx \int\limits_0^{L_\alpha}\varphi_\alpha^\prime(s)\times\left(\frac{w-y_\alpha}{\left\vert w-y_\alpha\right\vert^3}\right)ds \\
	&= \eta_\alpha \times\left(\frac{w-y_\alpha}{\left\vert w-y_\alpha\right\vert^3}\right)
\end{align*}
with 
$ \eta_\alpha = \int\limits_0^{L_\alpha}\varphi_\alpha^\prime(s)~ds$ the orientation of the approximated coil.
Thus, further ignoring multiplicative constants, we can write an idealized measurement $b_\alpha$ related to the activation $y_\alpha$ and the vector $\eta_\alpha$ as 
%\begin{align*}
$$	b_\alpha(\sigma)= \sigma_n \cdot\!\int_\Omega\!\!\left(\frac{3(\sigma_x-w)\otimes(\sigma_x-w)}{\left\vert \sigma_x-w\right\vert^5}-\frac{\mathbb{I}}{\left\vert \sigma_x-w\right\vert^3}\right)\!\!\left( \eta_\alpha \times\left(\frac{w-y_\alpha}{\left\vert w-y_\alpha\right\vert^3}\right)\right)\!\!c(w)~\!dw. $$
%\end{align*} 
The formula for the measurement can be brought into a more compact notation by employing the fundamental solution $\gamma(x) = \frac{1}{4 \pi \vert x \vert}$ of the Laplace equation (cf. \cite[Chapter 2.2]{evans1998partial}) and noticing that 
$$ \nabla_x\gamma(x) = - \frac{1}{4\pi}\frac{x}{\vert x\vert^3} $$
and
$$ \nabla_x\nabla_x\gamma(x) = \frac{1}{4\pi}\frac{3x\otimes x}{\vert x\vert^5}-\frac{\mathbb{I}}{\vert x\vert^3}.$$
%$$ \frac{3(\sigma_x-w)\otimes(\sigma_x-w)}{\left\vert \sigma_x-w\right\vert^5}-\frac{\mathbb{I}}{\left\vert \sigma_x-w\right\vert^3} = 4\pi \nabla_{(\sigma_x-w})\nabla_{(\sigma_x-w)} \gamma(\sigma_x-w) $$
%respectively 
%$$ \frac{w-y_\alpha}{\left\vert w-y_\alpha\right\vert^3} = - 4 \pi \nabla  \gamma(w-y_\alpha) .$$
Hence, we obtain
%\begin{align*}
$$	b_\alpha(\sigma)=  16 \pi^2 \sigma_n \cdot \int_\Omega \nabla_{(\sigma-w)}\nabla_{(\sigma-w)} \gamma(\sigma_x - w) \cdot
	\left( \eta_\alpha \times \nabla_{(w-y_\alpha)} \gamma(w-y_\alpha) \right) c(w) ~dw.
$$
%\end{align*}
Note that the following identities hold for the fundamental solution $\gamma$:
$$ \nabla_{(x_1-x_2)}\gamma(x_1-x_2)=\nabla_{x_1}\gamma(x_1-x_2)=-\nabla_{x_2}\gamma(x_1-x_2). $$
By using these identities, integration by parts and the compact support of $c$ in $\Omega$ we get
\begin{align*}
	 b_\alpha(\sigma)&= -16 \pi^2 \sigma_n \cdot \nabla_{\sigma_x} \int_\Omega  \nabla_w\gamma(\sigma_x - w) \left(  c(w) \eta_\alpha \times \nabla_w \gamma(w-y_\alpha) \right) ~dw\\
&= -16 \pi^2 \sigma_n \cdot \nabla_{\sigma_x} \int_\Omega  \gamma(\sigma_x - w) \nabla_w \cdot \left(  c(w) \eta_\alpha \times \nabla_w \gamma(w-y_\alpha) \right) ~dw\\
&= -16 \pi^2 \sigma_n \cdot \nabla_{\sigma_x} U(\sigma_x;y_\alpha,\eta_\alpha) .
\end{align*}
Using the properties of the fundamental solution $\gamma$ we can characterize $U(\cdot;y_\alpha,\eta_\alpha)$ as the unique solution of
\begin{equation} \label{eq:uequation}
	- \Delta U = \nabla \cdot 
\left(  c A \right) \qquad \text{in } \R^3,
\end{equation}
decaying at infinity, with the activation vector field 
$$ A(x) = \eta_\alpha \times \nabla  \gamma(x-y_\alpha). $$

In the following we further assume that the coils can be arranged around a hypersurface $\Gamma \subset \partial \Omega$ and that coils with three different orientations spanning the whole $\R^3$ are available. This means that one indeed measures  $16\pi^2\nabla U(\cdot;y_\alpha,\eta_\alpha)$ effectively.
 Finally, ignoring known scaling constants we can assume that $\eta_\alpha$ is normalized and the measurement corresponds directly to $\nabla U$ on $\Gamma$. This leads to the following idealized problem, that will be the basis of further analysis:

\begin{quote}
{\bf Idealized Inverse Problem:}\\
Given measurements of $\nabla U(\cdot;y,\eta)$ on $\Gamma$ for a set of activations $(y,\eta) \in \Theta \subset \Sigma$, where $U$ is the solution of Equation \eqref{eq:uequation}, determine the magnetic particle density $c$ compactly supported in $\Omega$.
\end{quote}

We see that in this setting the problem shares similarities to inverse source problems (cf. \cite{Isakov1}). Indeed basic unique continuation results for the Laplace equation from Cauchy data on a hypersurface (or even a stronger result from the knowledge of $|\nabla U|$ only, cf. \cite[Lemma 2.1.1]{Isakov1}) show that from such data $U$ is uniquely determined in $\R^3 \setminus \Omega$. Another analogy can be drawn to inverse problems in fluorescence tomography (cf. \cite{egger2010forward}). By rewriting the activation via the solution of another Poisson equation, we thus have the system
\begin{align}
	\begin{split}
	- \Delta V &= \delta_{y} \\
	-\Delta U &= \nabla \cdot ( c (\eta \times \nabla V)) = - \nabla \cdot ( c (\nabla \times (V \eta))).
	\end{split}
\end{align}
Apart from the fact that elliptic operators on bounded domains with more complicated coefficients are used in fluorescence tomography, the key difference is the way of activation. In fluorescence tomography the right-hand side in the second equation is of the form $c V$, while here we find a non-scalar version mediated by the effective coil orientation $\eta$. 

We mention that an analogous formulation is possible for the original forward model, however there is no equivalent of the scalar potential $V$ and we need to write an equation for a vector field $W$ corresponding to $-V \eta$ in the above formulation. Noticing that
\begin{align*}
	\frac{1}{4\pi} \int\limits_0^{L_\alpha}\varphi_\alpha^\prime(s)\times\left(\frac{w-\varphi_\alpha(s)}{\left\vert w-\varphi_\alpha(s)\right\vert^3}\right)~ds 
	&= - \int\limits_0^{L_\alpha}\varphi_\alpha^\prime(s)\times\left(\nabla_w\gamma(w-\varphi_\alpha(s))\right)~ds \\
	&= - \nabla_w \times  \int\limits_0^{L_\alpha}\varphi_\alpha^\prime(s) \gamma( w-\varphi_\alpha(s))~ds
\end{align*} 
since $\nabla_w \times(\varphi_\alpha^\prime(s))=0$. Then we can write
\begin{align}
	\begin{split}
		- \Delta W &= \epsilon_{\alpha} \\
	-\Delta U &= \nabla \cdot ( c (\nabla \times W))
	\end{split}
\end{align}
with the vectorial distribution 
$$ \epsilon_{\alpha}: \psi \mapsto \int\limits_0^{L_\alpha}\varphi_\alpha^\prime(s) \psi( \varphi_\alpha(s))ds. $$

\subsection{Idealized Dipole Model}

The asymptotic analysis in the previous section is based on the implicit assumption that $\eta_\alpha \neq 0$, otherwise no activation is left at leading order. Since $\eta_\alpha = 0$ may happen in practice, in particular for any coil represented by a closed curve, we further discuss this case in the following. The appropriate model arises from the first order expansion of 
\begin{align*}
	\frac{w-\varphi_\alpha(s)}{\left\vert w-\varphi_\alpha(s)\right\vert^3} \approx \left(\frac{w-y_\alpha}{\left\vert w-y_\alpha\right\vert^3}\right) - \nabla \left(\frac{w-y_\alpha}{\left\vert w-y_\alpha\right\vert^3}\right) (\varphi_\alpha(s) - y_\alpha).
\end{align*}
With $\eta_\alpha = \int_0^{L_\alpha}\varphi_\alpha^\prime(s)ds = 0$ this leads to
\begin{align*}
	\int\limits_0^{L_\alpha}\varphi_\alpha^\prime(s)&\times\left(\frac{w-\varphi_\alpha(s)}{\left\vert w-\varphi_\alpha(s)\right\vert^3}\right)ds\\
	%&\approx \underbrace{\int\limits_0^{L_\alpha}\varphi_\alpha^\prime(s)ds}_{=\eta_\alpha}\times \left( \frac{w-y_\alpha}{\left\vert w-y_\alpha\right\vert^3}\right) - \int\limits_0^{L_\alpha}\varphi_\alpha^\prime(s)\times \left( \nabla \left(\frac{w-y_\alpha}{\left\vert w-y_\alpha\right\vert^3}\right) (\varphi_\alpha(s) - y_\alpha)\right) ds \\
	&= - \int\limits_0^{L_\alpha}\varphi_\alpha^\prime(s)\times \left( \nabla \left(\frac{w-y_\alpha}{\left\vert w-y_\alpha\right\vert^3}\right) (\varphi_\alpha(s) - y_\alpha)\right) ds \\
	&= \nabla \times ( M_\alpha \nabla \gamma(w-y_\alpha))\\
	&= \nabla \times (\nabla \cdot ( M_\alpha \gamma(w-y_\alpha))) 
\end{align*}
with the matrix 
$$ M_\alpha  = 4\pi  \int\limits_0^{L_\alpha}\varphi_\alpha^\prime(s)\otimes (\varphi_\alpha(s) - y_\alpha)~ds . $$

With the notations of the previous section we can also rewrite a slightly changed formula for the activation and obtain the forward solution $U$ now with the changed activation model
\begin{align}
	\begin{split}
		- \Delta W &= \nabla \cdot (M_\alpha \delta_y) \\
		-\Delta U &= \nabla \cdot ( c (\nabla \times W)),
	\end{split}
	\label{eq:uequation2}
\end{align}
which actually corresponds to a magnetic dipole at $\delta_y$. 
\begin{quote}

{\bf Idealized Dipole Inverse Problem:}\\
Given measurements of $\nabla U(\cdot;y,\eta)$ on $\Gamma$ for a set of activations $(y,M) \in \Theta \subset (\R^3\setminus\Omega) \times \mathbb{R}^{3\times 3}$, where $U$ is the solution of Equation \eqref{eq:uequation2}, determine the magnetic particle density $c$ compactly supported in $\Omega$.

\end{quote} 
The analysis of this inverse problem shares many similarities with the idealized model above and will hence not be discussed in detail in the following. However, we will use the two dimensional version of the dipole model for the first set of numerical experiments.

\section{The Inverse Problem of MRXI}\label{sec:inverseproblem}

In the following we further outline some properties of the inverse problem in MRXI, discuss possible variational regularizations and constraints, and finally provide a more detailed analysis of the idealized inverse problem.

\subsection{Ill-Posedness of the Inverse Problem} 

The modeling above provides an operator $\mathbf{K}$ for a given activation pattern $(\omega)_\alpha$. With knowledge of the measurements $g$ we can denote an operator linear in the particle distribution $c$. Then we have a standard linear inverse problem in the form of the operator equation
\begin{align*}
\mathbf{K}c=g.
\end{align*}
%where $g$ defines measured magnetic response and depends on a specific $\sigma$. Here $\mathbf{K}$ and $c$ are necessary to calculate $g$, but in general $c$ is the unknown factor whereas the system operator $\mathbf{K}$ can be modeled and $g$ is obtained by measurement devices. Then the general approach is
%\begin{align*}
%\text{find}\quad c\quad\text{s.t.}\quad \mathbf{K}c=g.
%\end{align*}
As mentioned above we can see each integral operator $\mathbf{K}_\alpha$ as a Fredholm operator of the first kind (see Equation \eqref{eq:fredholmformulation}) with bounded kernel. Then the operator $\mathbf{K}_\alpha\colon\mathcal{L}^2(\Omega)\rightarrow\mathcal{L}^2(\Sigma)$ is bounded and compact (cf. \cite{Engl2013}). If the set of all coils $\mathcal{A}$ is finite it directly follows that $\mathbf{K}$ is compact as well on the corresponding product topology and hence the inverse problem is ill-posed. 
Similar arguments also hold for other versions of the index set $\mathcal{A}$ discussed before. We mention that in the realistic case of both activation coils and sensors being outside the region of interest, the operator $\mathbf{K}$ is a Fredholm integral operator with analytic kernel, hence the inverse problem is severely ill-posed.

\subsection{Variational Regularization}

In order to compute stable approximations of the solution despite the ill-posedness, we use the popular approach of variational regularization methods, i.e. we look for
\begin{equation}
	c^\ast=\arg\min_c \frac{1}{2}\Vert \mathbf{K}c-g\Vert^2_2 + \alpha R(c).\label{eq:generalvariationalmodel}
\end{equation}
%\begin{align*}
%&c^\dagger\quad\text{is lsq solution of}\quad \mathbf{K}c = g\quad\text{if}\quad\Vert \mathbf{K}c^\dagger-g\Vert^2_2\leq\Vert \mathbf{K}c-g\Vert^2_2\ \forall c\\
%\text{which correspond to}&\\
%&
%c^\dagger=\arg\min_c\frac{1}{2}\Vert \mathbf{K}c-g\Vert^2_2
%\end{align*}
%The later representation is also called a \textit{variational model }\todo{stimmt das? ist das schon ein variations model, auch ohne reg?} of an inverse problem. 
In general this is a flexible solution approach, allowing to include prior knowledge and further constraints to the solution of the original problem. For example, the assumption of a smooth solution can be included as prior knowledge leading to
\begin{align}
\arg\min_c\frac{1}{2}\Vert \mathbf{K}c-g\Vert^2 + \alpha\Vert\nabla c\Vert^2_{\mathcal{L}^2(\Omega)}
\end{align}
as an explicit example commonly known as first-order Tikhonov regularization. Since Tikhonov regularization and methods producing similar results are widely used in practice, we will consider it as state-of-the-art method and use it for comparison with the reconstruction method proposed in this paper. In general, the choice of a penalization term depends on the considered problem and often originates from physical principles or constraints and includes system relevant properties.

For MRXI, particles are distributed into the region of interest. This region of interest may consist of various materials with individual physical properties that yield different characteristic densities. Due to the limit resolution of MRXI we do not expect to resolve local variations in the density, but rather focus on reconstructing sharp edges between different kinds of tissue and assume that the magnetic nanoparticles will distribute homogeneously in a certain tissue.
% of the same density, as we expect to use this technique for organs or other tissues. This leads to the assumption that the particle distribution is homogeneous with sharp edges between different kinds of tissue. 
It is well known that the special property of constant regions and (mainly) sharp edges is supported using Total Variation regularization (cf. \cite{chambolle2010introduction,burger2013guide}). Hence, we incorporate the Total Variation seminorm
\begin{equation}
	TV(c) = \sup_{\psi \in C_0^\infty(\Omega)^3, \Vert \psi \Vert_\infty \leq 1} \int_\Omega c(x) \nabla \cdot\psi(x)~dx
\end{equation}
as a regularization term in our variation model \eqref{eq:generalvariationalmodel}.
%The regularizer will return small values if the argument has only a few jumps in intensity (in our case density). That makes the Total Variation regularizer perfect for our assumption that particles will distribute equally in one sort of tissue.
In addition we incorporate the natural constraint that
%For MRXI we use magnetic nanoparticles and reconstruct a three-dimensional density map of these particles as a solution of our problem. Since 
a density function is, from a physical point of view, a nonnegative function. Thus we restrict the minimization to nonnegative functions. In order to incorporate a constraint into the variational model we employ
% to represent the particle position in space, we are looking for a positivity constraint on the solution.
%Our variational model is formulated so that we search for a minimum of a functional, that consist of a least squares data-fidelity. Now we extent the functional by adding an operator $\chi$ that is defined 
the characteristic function of a convex set $C$, i.e. 
\begin{align}
\chi_C(x) = \left\{\begin{array}{cl}
0\qquad&\text{if}\quad x\in C\\
\infty&\text{else}
\end{array}\right..
\end{align}
To implement a nonnegativity constraint $c\geq0$ on the particle distribution $c$ we define $C=\{x\vert x\geq0\}$ and consider $\chi_C(c)$,  or simply write $\chi_{\geq 0}(c)$.
%Since we solve the variation problem as a minimization problem, this operator will prevent minima (therefore possible solution of our problem) that have negative entries, because negative entries would set $\chi$ to $\infty$.\\+
This leads to the variational model
\begin{align}
c^\ast \in \arg\min_c\frac{1}{2}\Vert \mathbf{K}c-g\Vert^2 + \alpha\operatorname{TV}(c) + \chi_{\geq 0}(c)\label{eq:variationModel}
\end{align}

Indeed this problem is well-defined, i.e. a nonnegative minimizer $c^\ast$ exists for any $\alpha > 0$ in 
$$  BV(\Omega) = \{ c \in \mathcal{L}^1(\Omega)~|~ TV(c) < \infty ~\}.
$$
This follows by minor modifications from standard existence results for Total Variation regularization (cf. e.g. \cite{Acar1994,burger2013guide}), using the compact embedding of $BV(\Omega)$ into $\mathcal{L}^1(\Omega)$ and the boundedness of the forward operator $\mathbf{K}$ on the latter space.

Instead of the variational method, which is known to produce a rather strong bias (a loss of contrast in the case of total variation regularization, cf. \cite{benning2013ground,burger2013guide}) we can employ the Bregman iteration as an iterative regularization method (cf. \cite{Osher2005}). This means for $c^0=0$, $g^0=g$ and $\alpha$ fixed and large, we iteratively compute a sequence of reconstructions
	\begin{align}
c^{k+1} &\in \arg\min_c\frac{1}{2}\Vert \mathbf{K}c-g^k\Vert^2 + \alpha\operatorname{TV}(c) + \chi_{\geq 0}(c) \\
g^{k+1} &= g^k + g - Kc^{k+1}. 
\end{align}
The regularizing effect arises from an appropriate stopping of the iteration, where previous analysis indicates that $\alpha$ times the number of Bregman iterations corresponds to the regularization parameter in the regularization method, however with a reduction of the bias compared to the variational method (cf. \cite{benning2013ground,burger2013guide}). If an estimate of the size of the noise is available, one can easily choose a stopping index by the discrepancy principle or a similar method.

\subsection{Identifiability Analysis of the Idealized Problem}

We finally turn to a more detailed analysis of the idealized problem formulated in Subsection \ref{ssec:idealmodel}. In particular we are interested in the uniqueness or possible non-uniqueness in the determination of the density $c$ from different activation strategies. We will consider two extreme cases in order to understand when uniqueness can hold or fail in such a problem:
\begin{itemize}

\item {\bf Full activation:} in this scenario we assume that activation is carried out at any $y \in {\cal O} \subset\R^3 \setminus \Omega$ for an open set ${\cal O}$, with two linearly independent orientations $\eta_1(y)$ and $\eta_2(y)$ for every $y$. 

\item {\bf Far-field activation:} in this scenario we assume that activation is carried out at any $y \in \partial B_R(0)$ in the limit $R \rightarrow \infty$, again with two linearly independent orientations $\eta_1(y)$ and $\eta_2(y)$ for every $y$.
\end{itemize}

We start with a rather standard reciprocity principle that holds for any kind of activation:
\begin{lem} \label{lem:reciprocity} 
Let $c_1, c_2 \in \mathcal{L}^1(\Omega)$ be nonnegative densities with compact support in $\Omega$, such that the potentials $U_i$ related to $c_i$ satisfy 
\begin{equation}
	\nabla U_1(\cdot;y,\eta) = \nabla U_2(\cdot;y,\eta) 
\end{equation}
on a $C^1$-surface $\Gamma$ outside $\Omega$. Then for all $z \in \mathbb{R}^3 \setminus \Omega$ and $\tilde c = c_1-c_2$ the identity
\begin{align}
\begin{split}
	0&= \int_\Omega \tilde c(x) \nabla \gamma(x-z) \cdot ( \eta \times \nabla \gamma(x-y)) ~dx\\
	& = \int_\Omega \tilde c(x) \nabla \gamma(x-y) \cdot ( \eta \times \nabla \gamma(x-z)) ~dx\label{eq:reciprocity}
\end{split}
\end{align}
holds, with $\gamma=\frac{1}{4\pi\vert x\vert}$ fundamental solution of the Laplace equation.
\end{lem}
\begin{proof}
First of all, due to the compact support of $\tilde c$ in $\Omega$, the potential $\tilde U = U_1 - U_2$ is a smooth harmonic function in $\mathbb{R}^3 \setminus \Omega$. As explained in Section 2, standard unique continuation for harmonic functions then implies $\tilde U \equiv 0$ in $\mathbb{R}^3 \setminus \Omega$. Moreover, $\tilde U$ is a weak solution of the Poisson equation on $\Omega$, more precisely for any function $\varphi \in C^\infty(\Omega)$ we have 
$$ - \int_\Omega \tilde U(x) \Delta \varphi(x) ~dx = \int_\Omega \tilde c (\eta \times \nabla \gamma(x-y)) \cdot \nabla \varphi(x)~dx. $$
Note that we do not require compact support of $\varphi$ due to the vanishing Cauchy data of $\tilde U$ and the compact support of $\tilde C$. Thus, for $z$ with positive distance to $\Omega$, $\varphi=\gamma(\cdot-z)$ is a suitable test function, harmonic in $\Omega$. This implies the second identity in \eqref{eq:reciprocity} for such $z$. Using the compact support of $\tilde c$ in $\Omega$ it is straightforward to extend to all $z \in \mathbb{R}^3$ by an approximation argument.
The first identity follows from an elementary vector identity.
\end{proof} 

Now let us further characterize the type of measurements we find in the different application scenarios: 
\begin{lem} \label{lem:fullactivation}
Let the set of activations be chosen according to the full activation scenario. Then with the assumptions of Lemma \eqref{lem:reciprocity} we have
\begin{equation}
	0 = \int_\Omega \tilde c(x) \frac{x-y}{|x-y|^6} \cdot ( \eta_i(y) \times  e) ~dx \label{eq:fullactivationidentity}
\end{equation}
for all $y \in {\cal O} \subset \mathbb{R}^3 \setminus \Omega$, $i=1,2$, and $e \in \mathcal{S}^2$.
\end{lem} 
\begin{proof}
Given $y \in \mathbb{R}^3 \setminus \Omega$ with positive distance to $\Omega$ and $e \in \mathcal{S}^2$, choose $z_e^\epsilon:= y -\epsilon e \in \mathbb{R}^3 \setminus \Omega$ 
for $\epsilon$ sufficiently small. From \eqref{eq:reciprocity} and the definition of $\gamma$ we see that 

\begin{align*}
0 & = \lim_{\epsilon \rightarrow 0} \left( \frac{16\pi^2}{\epsilon} \int_\Omega \tilde c(x) \frac{x-y}{|x-y|^3} \cdot ( \eta_i(y) \times \frac{x-z_e^\epsilon}{|x-z_e^\epsilon|^3} ) ~dx \right)\span\\
& = \lim_{\epsilon \rightarrow 0} \left( \frac{16\pi^2}{\epsilon} \int_\Omega \tilde c(x) \frac{1}{|x-y|^3|x-z_e^\epsilon|^3}\right.\span\\
&& \left((x-y) \cdot \eta_i(y) \times (x-y) + (x-y) \cdot \eta_i(y) \times (x+\epsilon e)\right)~dx\bigg)\\
& = \lim_{\epsilon \rightarrow 0} \left( 16\pi^2 \int_\Omega \tilde c(x) \frac{1}{|x-y|^3|x-z_e^\epsilon|^3} \left((x-y) \cdot \eta_i(y) \times (\frac{x}{\epsilon}+ e)\right)~dx \right)\span
\end{align*}
leads to the right-hand side in \eqref{eq:fullactivationidentity} using orthogonality of $(x-y)$ to $\eta_i(y) \times (x-y)$ and for $\epsilon\rightarrow 0$.
\end{proof}

\begin{lem} \label{lem:farfieldactivation}
Let the set of activations be chosen according to the far-field activation scenario. Then with the assumptions of Lemma \eqref{lem:reciprocity} we have
\begin{align*}
&\lim_{|y| \rightarrow \infty} |y|^2	\int_\Omega \tilde c(x) \nabla \gamma(x-z) \cdot ( \eta \times \nabla \gamma(x-y)) ~dx\\
&\qquad= -
	\int_\Omega \tilde c(x) \frac{x-z}{|x-z|^3} \cdot ( \eta_i(y) \times  \frac{y}{|y|}) ~dx   
\end{align*}
and thus,
\begin{equation} \label{eq:Coulomb}
\int_\Omega \tilde c(x) \frac{1}{|x-z|} \cdot ( \eta_i(y) \times  \frac{y}{|y|}) ~dx  =0	
\end{equation}
for all $z \in \mathbb{R}^3 \setminus \Omega$. 
\end{lem}
\begin{proof}
The first identity for the limit follows from a  straightforward computation of the limit. Together with \eqref{eq:reciprocity} we obtain 
$$ \nabla_z \int_\Omega \tilde c(x) \frac{1}{|x-z|} \cdot ( \eta_i(y) \times  \frac{y}{|y|}) ~dx  =0	 $$
for all $z \in \mathbb{R}^3 \setminus \Omega$. The asymptotic decay of the Coulomb potential for $|z| \rightarrow 0$ implies that there is no constant when integrating the gradient, i.e. \eqref{eq:Coulomb} holds.
\end{proof}

Since the measurement \eqref{eq:Coulomb} is known to be insufficient to determine $\tilde c$, which would amount exactly to the inverse source problem for the Poisson equation (cf. \cite{Isakov1}), we have to expect non-uniqueness in this activation scenario. Note that, strictly speaking, we have not given a non-uniqueness argument, since \eqref{eq:reciprocity} is only a sufficient condition and we have not used the nonnegativity of the densities $c_i$, but it appears natural that this case cannot suffice to uniquely determine the magnetic particle density.

In the case of full activation we have a different picture, there we can determine a suitable Riesz potential instead of the Coulomb potential, allowing to give a uniqueness result:

\begin{thm}
Let the set of activations be chosen according to the full activation scenario. Then the measurements of $\nabla U(\cdot;y,\eta_i(y))$, $i=1,2$, uniquely determine a nonnegative density $c \in \mathcal{L}^1(\Omega)$ with compact support.
\end{thm}
\begin{proof}
According to Lemma \eqref{lem:reciprocity} and Lemma \eqref{lem:fullactivation} it suffices to show that \eqref{eq:fullactivationidentity} implies $\tilde c \equiv 0$ in $\Omega$. First of all we see that with two linearly independent vectors $\eta_i(y)$   we obtain
$$ \left\{ \frac{\eta_i(y) \times  e}{|\eta_i(y) \times  e|} ~|~ e \in \mathcal{S}^2 \right\} = \mathcal{S}^2, $$
i.e. we can achieve arbitrary directions in \eqref{eq:fullactivationidentity}. Hence, we find
$$ \nabla_y \int_\Omega \tilde c(x)   \frac{1}{|x-y|^4}  ~dx  = - 4 \int_\Omega \tilde c(x) \frac{x-y}{|x-y|^6}  ~dx  = 0.  $$ 
Noticing its decay at infinity, the Riesz-Potential
$$ R_4(y) = \int_\Omega \tilde c(x)   \frac{1}{|x-y|^4}  ~dx  $$ 
consequently vanishes for all $y \in \mathcal{O}$. Since $R_4$ is analytic outside $\Omega$ and ${\mathcal O}$ unique analytic continuation implies that it vanishes in $\mathbb{R}^3 \setminus \Omega$. This is well-known to imply $\tilde c \equiv 0$ in $\Omega$ (cf. \cite{Isakov1,Isakov2}).
\end{proof} 

\section{Discrete Forward Model and Numerical Solution}
\label{sec:discreteForwardModel}

In the following we discuss an appropriate discretization of the inverse problem, its implementation and finally the numerical optimization of the discretized variational model.

\subsection{Discrete Operator}
In the following we provide a discretization of the MRXI forward operator \eqref{eq:fredholmformulation}. For our purposes we set $ \Omega\defgr[0,1]^3 $ as a domain for the magnetic nanoparticles. We assume that activation fields and sensors have finite but positive distance to $\Omega$. Therefore we define $ \Omega_0 \defgr [-r,1+r]^3\backslash\Omega $ that holds both coils and sensors with $0<r<\infty$.

In the discrete setting the curve $\varphi_\alpha$ that reassembles the conductor coil of an activation $\alpha$ with length $L_\alpha$ is cut into $l$ piecewise linear segments. Then the $k$-th segment is defined as
%\begin{align*}
$$\varphi_{\alpha,k}\colon\left[(k-1)\frac{L_\alpha}{l},k\frac{L_\alpha}{l}\right] \rightarrow\Omega_0, \quad 
t \mapsto a_k + t(b_k-a_k) $$
%\end{align*}
with $a_k, b_k\in\Omega_0$. Then the approximated conductor path $\tilde{\varphi}_\alpha$ is defined as
$ \tilde{\varphi}_\alpha \defgr \bigoplus\limits_{k=1}^l\varphi_{\alpha,k}$
%\todoi{alternativ:}
%
%\begin{align*}
%\tilde{\varphi}_\alpha(t)\colon\left[0,L_\alpha\right]&\rightarrow\Omega_0\\
%t&\mapsto\left\{\begin{array}{lll}
%\varphi_{\alpha,1}(0)&&\text{for }
%\begin{array}{l}
%t=0
%\end{array}\\
%\varphi_{\alpha,k}\left(t-(k-1)\tfrac{L_\alpha}{l}\right)&\quad&\text{for }
%\begin{array}{l}
%t\in \left((k-1)\tfrac{L_\alpha}{l}, k\tfrac{L_\alpha}{l}\right], \\
%\qquad k = 1, \ldots, l
%\end{array}
%\end{array}\right.
%\end{align*}
where $\bigoplus$ defines the concatenation of $\varphi_{\alpha,k}$ for $k=1,\ldots,l$. In general we require that $ \lim\limits_{l\rightarrow\infty} \tilde{\varphi}_\alpha = \varphi_\alpha $ to provide a proper discretization of the activation coil. Then the $k$-th segment provides a magnetic field in $w\in\Omega$ (cf.\cite{Hanson2002}):
\begin{align}
\mathbf{B}_{\alpha,k}^\textbf{coil}\colon\Omega&\rightarrow\R^3\label{eq:discreteActivation}\\
w&\mapsto\vartheta\frac{\vert a_k-w\vert+\vert b_k-w\vert}{\vert a_k-w\vert\vert b_k-w\vert}\frac{(a_k-w)\times(b_k-w)}{\vert a_k-w\vert\vert b_k-w\vert+(a_k-w)\cdot(b_k-w)}.\nonumber
\end{align}

Thus the magnetic field induced by a $l$-segmented activation coil is given by
$$\mathbf{B}_\alpha^\textbf{coil}(w)=\sum_{k=1}^{l}\mathbf{B}_{\alpha,k}^\textbf{coil}(w).$$
Now we define a disjunct decomposition $(\Omega_k)_{k=1\ldots N}$ of $\Omega$ as well as a set of midpoints $W = \{w_k\}_{k=1\ldots N}$ and demand that $\overline{\cup_k\Omega_k} = \Omega$ holds. Then the discretized particle distribution $c$ is defined by 
\begin{equation}
%\begin{align}
c\colon W \rightarrow\R, 
w_k  \mapsto \frac{1}{\vert\Omega_k\vert}\int\limits_{\Omega_k}c(w)~dw.
%\end{align}
\end{equation}
Thus, on a grid $W$ with $N$ nodes, the continuous particle distribution $c\in\mathcal{L}^2(\Omega)$ can be understood as an element $c\in\R^N$ in the discrete setting. 

For a given but fixed measurement $\sigma$, Equation \eqref{eq:fredholmformulation} can be translated into a linear operator
\begin{equation}
%\begin{align}
\mathbf{K}_{\alpha, \sigma}\colon \R^N \rightarrow\R, \quad
c\mapsto\sum\limits_{k=1}^N\mathbf{k}_\alpha(w_k,\sigma)c(w_k)\label{eq:discreteForwardOp}
%\end{align}	
\end{equation}
where $\mathbf{k}_\alpha$ is defined as seen in Equation \eqref{eq:mrxikernel}. However, in this case we are using the discrete approximation of $\mathbf{B}_\alpha^\textbf{coil}$ (Equation \eqref{eq:discreteActivation}) instead. In the end $\mathbf{K}_{\alpha, \sigma}c = g \in\R$ provides a measurement of the full amplitude of the relaxation induced by coil activation $\alpha$. 

\subsection{Simplified Coil Activation}\label{ssec:discretesimplifiedcoil}

We see that a realistic implementation of the activation coil using piecewise linear sections leads to Equation \eqref{eq:discreteActivation}.
As a result the number of calculations, and therefore the complexity of the computation speed, increases with higher precision modeling of the coil.
However, in Subsection \ref{ssec:idealmodel} we have shown that the coil can be approximated in the limit by a small coil, maintaining its magnetic properties, resulting in the idealized model.
Thus, the excitation coil can be implemented as a magnetization peak $\eta_\alpha\delta$.
This way both, coil excitation and dipole relaxation, is described by Equation \eqref{eq:measmagnetization}, however the coil activation fields are not restricted by any measuring direction $\sigma_n$.
This leads to a much more lightweight computation for the coil activation.

\subsection{Matrix Assembly: Sensors and Activation Patterns}

In an experimental setup the amplitude of the particle relaxation is acquired in multiple points $\sigma_1, \ldots, \sigma_s$ where $s$ provides the number of sensors in the system. Due to the linearity in $c$, Equation \eqref{eq:discreteForwardOp} can be translated in a matrix representation ${\mathbf{K}}_\alpha \in \R^{s\times N}$ combining all measuring points $\sigma_1, \ldots, \sigma_s$:
$$ 
{\mathbf{K}}_\alpha\defgr\left[\begin{array}{c}
\mathbf{K}_{\alpha, \sigma_1}\\
\vdots\\
\mathbf{K}_{\alpha, \sigma_s}
\end{array}\right].
$$
Then for multiple subsequent activations $\alpha_1\ldots\alpha_r$ this leads to a fully discretized forward operator matrix ${\mathbf{K}}\in\R^{M\times N}$ with $M=rs$:
\begin{align}
	\begin{split}		
		{\mathbf{K}}\defgr\left[\begin{array}{c}
			{\mathbf{K}}_{\alpha_1}\\
			\vdots\\
			{\mathbf{K}}_{\alpha_r}
		\end{array}\right].
	\end{split}\label{eq:fulldiscreteoperator}
\end{align}
For a generic variational approach it suffices to have the discretized forward operator ${\mathbf{K}}$, respectively its matrix representation. For advanced numerical methods and in particular for future design of activation and measurement strategies it is however crucial to keep in mind its internal structure.
%
%\todoi{eigenltich müsste man hier die $\beta$ verwenden, die die gewichtung bei gleichzeitiger activierung definieren. Jedoch scheint mir das hier übersichtlicher zu sein}
%%

\subsection{Alternating Direction Method of Multipliers (ADMM)}

We defined our desired model for MRXI in  \eqref{eq:variationModel}:
\begin{align}
c^* \in \arg\min_c\frac{1}{2}\Vert \mathbf{K}c-g\Vert^2_2 + \alpha\Vert\nabla c\Vert_1 + \chi_{\geq 0}(c)\label{eq:variationModelII}
\end{align}
where $\Vert\nabla (\cdot)\Vert_1$ describes the discrete version of the Total Variation $\operatorname{TV}(\cdot)$.
To approach this problem we use ADMM to deduce an iterative scheme as suggested by \cite{Boyd2012}.

In the general formulation of ADMM we have convex functionals $\mathcal{D}\colon\R^N\to\R$ and $\Psi\colon\R^P\to\R$, linear operators $E\colon\R^N\to\R^M$ and  $H\colon\R^P\to\R^M$ as well as variables $c\in\R^N$, $v\in\R^P$ and $u\in\R$. Then a generalized representation of a minimization problem can be denoted as
\begin{equation}
	\mymin{c,v}\mathcal{D}(c)+\Psi(v) \qquad
\text{s.t.}\qquad Ec+Hv=u.\label{eq:admmgeneral}
\end{equation}
%\begin{align}
%\begin{matrix}
%\mymin{c,v}\mathcal{D}(c)+\Psi(v)\label{eq:admmgeneral}\\
%\text{s.t.}\qquad Ec+Hv=u.
%\end{matrix}
%\end{align}
This constrained problem can be expressed by the augmented Lagrangian
\begin{align*}	
\mathcal{L}_\rho\left(c, v, \lambda\right)=\mathcal{D}\left(c\right)+\Psi(v)+\lambda^T\left(Ec+Hv-u\right)+\frac{\rho}{2}\left\Vert Ec+Hv-u\right\Vert_2^2.
\end{align*}
Basically, the Lagrangian is a reformulation of the minimization problem, where a set $(c^\ast,v^\ast,\lambda^\ast)$ is an extrema of the Lagrangian if $c$ and $v$ are minimizer of the original problem. This leads to the following update scheme:
\begin{align}
c^{k+1}&=\myargmin{c}\mathcal{L}_\rho\left(c, v^k, \lambda^k\right)\label{eq:ADMM1}\\
v^{k+1}&=\myargmin{v}\mathcal{L}_\rho\left(c^{k+1}, v, \lambda^k\right)\label{eq:ADMM2}\\
\lambda^{k+1}&=\lambda^k+\rho\left(E(c^{k+1})+H(v^{k+1})-u\right)\label{eq:ADMM3}.
\end{align}

To transfer our variational model \eqref{eq:variationModelII} into the ADMM scheme \eqref{eq:admmgeneral} we have to define 
\begin{align*}
\mathcal{D}(c)&\defgr\frac{1}{2}\Vert \mathbf{K}c-g\Vert^2_2&v&\defgr\left(\begin{matrix}
\hat{v}\\v_+
\end{matrix}\right)\\
\Psi(v)&\defgr\alpha\Vert\hat{v}\Vert_1 + \chi_{\geq 0}(v_+)&u&\defgr 0\\
E&\defgr\left(\begin{matrix}
\nabla\\\mathbb{I}
\end{matrix}\right)&H&\defgr-\mathbb{I}.
\end{align*}
In this setting we can directly apply the ADMM scheme \eqref{eq:ADMM1}-\eqref{eq:ADMM3} for solution of the inverse problem of MRXI. Note that the major computational effort is in each iteration step is due to the first step, i.e. computing an update for $c$, which involves the solution of a large linear system with the full matrix $ \mathbf{K}^T \mathbf{K} + \rho E^T E$. There is strong future potential in speeding up the numerical linear algebra by exploiting the block structure of $\mathbf{K}$ in the future. For first tests in this paper we simply use a direct solver however. 
\section{Numerical Simulation}
In this section we present some simple results based on the idealized discrete setup that is described in Subsection \ref{ssec:discretesimplifiedcoil}.
The benefit of this simplified approach is that we can deduce a 2D framework for MRXI: the region of interest is now described by a 2D plane containing the space $\Omega$ for the magnetic nanoparticles and $\Omega_0$ for coils and sensors.
Due to the simplified model the coils preserve their magnetic properties in the 2D case and are represented by a position $y_\alpha$ and a magnetic moment $\eta_\alpha$.

\subsection{Setup}
\begin{figure}
	\centering
	\begin{subfigure}[t]{0.47\textwidth}\centering
		\includegraphics[trim=150 300 150 277,clip,width=\linewidth]{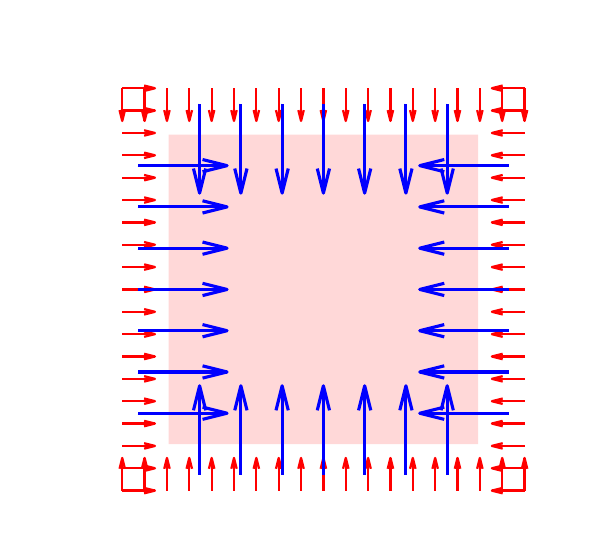}
		\subcaption{\footnotesize 2D Setup with coils, sensors and region of interest}
		\label{img:2dmrxisetip}
	\end{subfigure}
	\begin{subfigure}[t]{0.47\textwidth}
		\includegraphics[trim=150 297 150 280,clip,width=\linewidth]{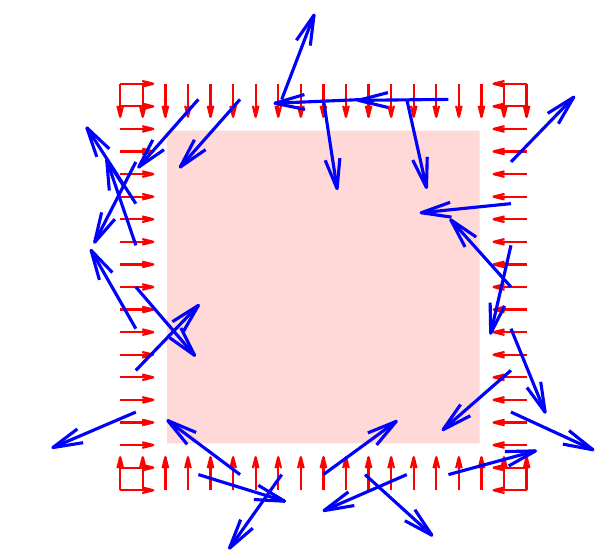}
	\subcaption{\footnotesize Setup with randomized coil orientation}
	\label{img:2dmrxisetiprandom}
	\end{subfigure}
	\caption{\footnotesize Visualization of the 2D MRXI setup. Activation coils are marked by blue arrows where the arrow origin marks the coil position. Sensors are defined by red arrows (again the sensor position is marked by arrow origin). $\Omega$ is marked by the pink area.}
\end{figure}

We set $\Omega$ as a square area. At each side seven activation coils are setup pointing towards the domain.
Furthermore we set $19$ measurement points at each of the four sides, again pointing towards the domain $\Omega$.
A visualization of this setup is shown in Figure \ref{img:2dmrxisetip}.
Additionally we consider a second setup with coil orientation as key difference: positions of sensors and coils are identical, however coil orientations are randomized.
An illustration of this alternative setup is shown in Figure \ref{img:2dmrxisetiprandom}. 
For both setups we consider consecutive coil activation.
This means we activate each coil separately and only once. 
Then for each of the $28$ coil activations the sensor system acquires $76$ measurements.
Thus this MRXI setup yields a measurement vector $g$ with $g \in \R^{28\cdot76} = \R^{2128}$.

For our simulations we consider three phantoms reassembling the particle distribution $c$: a simple P shape cluster of magnetic nanoparticles, a transverse slice of the widely used Shepp-Logan phantom and a simplified phantom of a tumor with a tube-like gap representing an intersecting vein (see first row of Figure \ref{img:results} for reference).

\subsection{Reconstructions}
\begin{figure}
	\centering
	\includegraphics[width=0.8\textwidth]{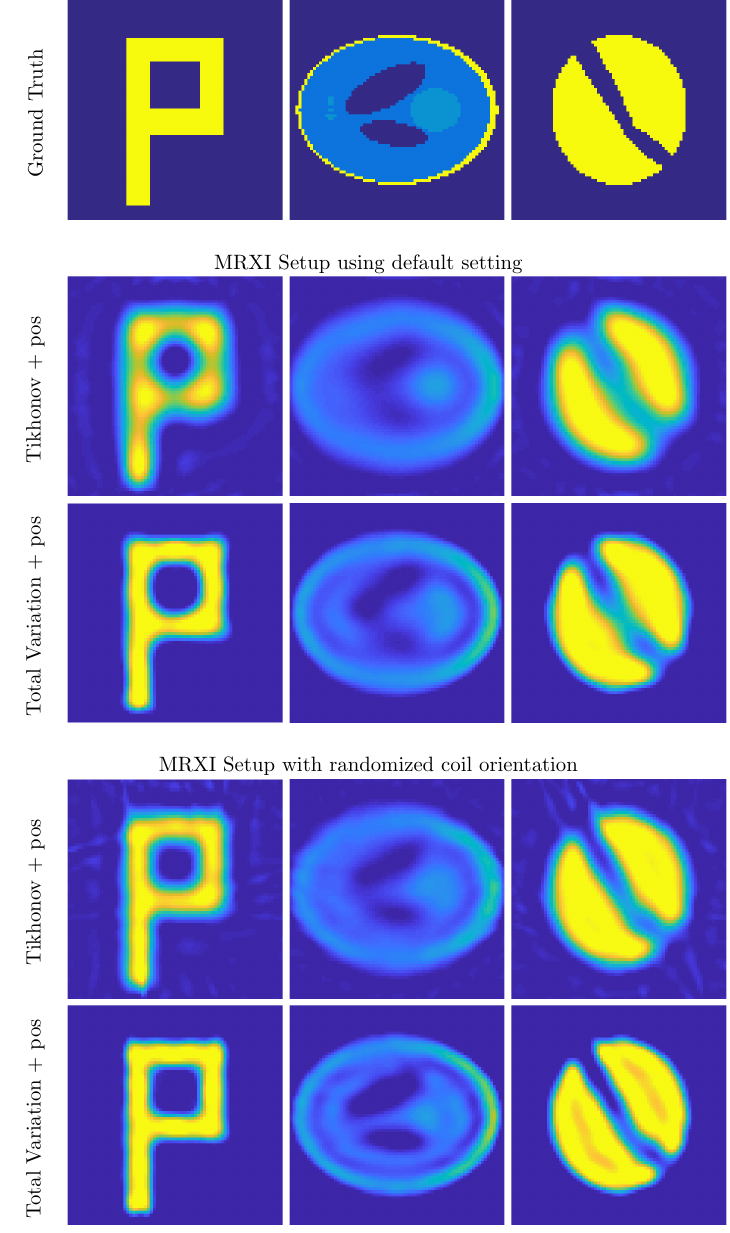}%  	
	%\includestandalone[width=0.8\textwidth]{images/results}%  	
	\caption{\footnotesize Particle distribution recovered using Tikhonov or Total Variation as regularizer. Positivity constraint is applied to all reconstructions.}
	\label{img:results}
\end{figure}
We aim for reconstructed resolutions of $75\times75$ pixels. The discrete forward operator (Equations \eqref{eq:discreteForwardOp} and \eqref{eq:fulldiscreteoperator}) can be implemented as a matrix $\mathbf{K}\in\R^{2128\times 5625}$. As a result our forward operator is highly underdetermined.
Next we will compare the variational model using Tikhonov regularization with positivity constraint
$$ c^\ast_{\operatorname{Tikh}} \in \arg\min_c\frac{1}{2}\Vert \mathbf{K}c-g\Vert^2 + \alpha\Vert c\Vert_2^2 + \chi_{\geq 0}(c) $$
and Total Variation regularization with positivity constraint (as of Equation \eqref{eq:variationModel})
$$ c^\ast_{\operatorname{TV}} \in \arg\min_c\frac{1}{2}\Vert \mathbf{K}c-g\Vert^2 + \alpha\operatorname{TV}(c) + \chi_{\geq 0}(c).$$

Since we limit ourselves to synthetic data we have to consider inverse crime to avoid optimal data fitting due to biased measurements. Therefore we simulate the measured data $g$ on a high resolution grid (in this case $197\times197$) and added noise with SNR 80dB.

The results shown in the second row of Figure \ref{img:results} illustrate that we can derive a rough estimate of the particle distribution using Tikhonov regularization.
One can distinguish the hole in the P, the empty spaces in the Shepp-Logan phantom as well as the particle free 'vein' in the simplified tumor phantom.
However, we can identify a general lack of contrast in regard to edges and areas with different particle distributions.
Furthermore we find small errors in the phantom background.

Optically we see improvements using Total Variation regularization.
The general shape of the P is much more accurate and the overall look is much cleaner.
This also holds for the Shepp-Logan phantom, where the inner contours are slightly better reconstructed, and the tumor phantom, where the vein can be identified as an intersecting and particle free area.

\begin{table}\centering
	\begin{tabular}{r|c|c|c}
		& P-shape & Shepp-Logan          & Tumor            \\ \hline\hline
		Tikh. + pos             & 0.115  & 0.100              & 0.097         \\
		TV + pos                & 0.210 & 0.158               & 0.187          \\ \hline 
	\end{tabular}
	\caption{\footnotesize SSIM values for reconstructions compared to ground truth using the default MRXI setting (see Figure \ref{img:2dmrxisetip} and \ref{img:results})}\label{t:ssim}
	\begin{tabular}{r|c|c|c}
		& P-shape & Shepp-Logan          & Tumor            \\ \hline\hline
		Tikh. + pos             & 0.155 & 0.139              & 0.136          \\
		TV + pos                & 0.257 & 0.222              & 0.212          \\ \hline
	\end{tabular}
	\caption{\footnotesize SSIM values for reconstructions compared to ground truth using the randomized coil orientation setting of MRXI (see Figure \ref{img:2dmrxisetiprandom} and \ref{img:results})}\label{t:ssimRandom}
\end{table}

Nevertheless the overall impression is that these basic reconstruction methods lack of precision and quality. This is also represented in the structure similarity (SSIM \cite{Wang2004}) index: Table \ref{t:ssim} contains the SSIM values for reconstructions shown in Figure \ref{img:results}. We see that for all given phantoms the quality of the reconstruction increases drastically by using Total Variation regularization. The overall similarity is not yet convincing. The lack of quality can be due to various reasons:

\begin{itemize}
	\item {\bf The choice of the regularization parameter:} note that the parameters are fixed for all phantoms. Therefore a proper choice of parameter in regard of the type of phantom will lead to more detailed reconstructions.
	\item {\bf Simplicity of Algorithm:} in our approach we naively reconstruct the system as it is and do not including any information in regard of coil or sensor placement. These information can contribute in the workflow of a more sophisticated algorithm.
	\item {\bf Trivial coil placement and activation pattern:} the choice of coil activations and positioning, as well as effective positioning using multiple coil activations at once can lead to more detailed reconstructions.
\end{itemize}

To attend the latter points we consider the MRXI setup with randomized coil orientation.
Then we see further improvement in the reconstructed particle distribution.
In particular edges of structures in the image are reconstructed more detailed.
Furthermore the smaller inner circle in the Shepp-Logan phantom can be seen more properly in both reconstructions methods, Tikhonov and Total Variation. 
In the tumor phantom the edge of the intersecting vein is much better localized, especially with Total Variation as regularization term.
This also reflects in the SSIM values for reconstruction in the randomized setup: Table \ref{t:ssimRandom} confirms the impression of higher quality reconstructions and endorses the advantages of Total Variation for the considered phantoms. 

Note that, in regard to Total Variation, which prioritizes solutions with sharp edges and homogeneous areas, the shown results are somehow expected. 
This is why Total Variation regularization is highly applicable for the kind of phantoms used in this study.

In summary we see some major improvements using Total Variation as regularization method and enforce randomized coil orientations.

\section{Conclusion and Outlook}

We have discussed the mathematical modeling and analysis of inverse problems in magnetorelaxometry imaging, which we have also put in the framework of inverse source problems for partial differential equations.
The variational regularization methods incorporating Total Variation regularization and a positivity constraint proposed in this paper indicate the potential improvement compared to simple linear reconstructions used in earlier investigations previously.
Our uniqueness analysis indicates that it will be beneficial to activate coils oriented in at least two different directions and possibly different radial position relative to the region of interest $\Omega$. 

The determination of efficient activation strategies is a key topic for future research.
Besides optimal choice of the locations for activation it also seems interesting to use optimal multiple activations.
This amounts to ideal choice of the parameters $\omega_\alpha^\beta$ in \eqref{eq:multipleactivation}.
Finally, the question of efficient sensor positioning arises, given a certain amount of available sensors.
Ideally, optimal experimental design is to be carried out for all parameters of the activation and measurement together, but one may expect that initial approaches will rather solve specific subproblems first.

On the reconstruction side the variational formulation of MRXI leads to a much more flexible model to find a suitable particle distribution.
At this point Total Variation regularization shows improvements in comparison to methods used in the literature, namely Tikhonov regularization.

For our reconstructions we added Gaussian noise to the measured data. However, the proposed model does not include noise in any sense, consequently the exact modeling has to be discussed in future studies.

\section*{Acknowledgements}
This work has been supported by the German Science Foundation (DFG) within the priority program CoSIP, project
CoS-MRXI (BA 4858/2-1, BU 2327/4-1). The authors thank Guillaume Bal (Columbia University) and Victor Isakov (Wichita State University) for useful and stimulating discussions.

MB would like to thank the Isaac Newton Institute for Mathematical Sciences, Cambridge, for support and hospitality during the programme Variational Methods for Imaging and Vision, where work on this paper was undertaken, supported by EPSRC grant no EP/K032208/1 and the Simons Foundation.
%\nocite{*}
%\printbibliography 
 
\bibliographystyle{abbrv}
\bibliography{bib/literature}

\begin{thebibliography}{10}

\bibitem{Acar1994}
R.~Acar and C.~R. Vogel.
\newblock Analysis of bounded variation penalty methods for ill-posed problems.
\newblock {\em Inverse Problems}, 1994.

\bibitem{Alexiou2011}
C.~Alexiou, R.~Tietze, E.~Schreiber, R.~Jurgons, H.~Richter, L.~Trahms,
  H.~Rahn, S.~Odenbach, and S.~Lyer.
\newblock Cancer therapy with drug loaded magnetic nanoparticles - magnetic
  drug targeting.
\newblock {\em Journal of Magnetism and Magnetic Materials}, 2011.

\bibitem{Baumgarten2014}
D.~Baumgarten, F.~Braune, E.~Supriyanto, and J.~Haueisen.
\newblock Plane-wise sensitivity based inhomogeneous eccitation fields for
  magnetorelaxometry imaging of magnetic of magnetic nanoparticles.
\newblock {\em Journal of Magnetism and Magnetic Materials}, 2014.

\bibitem{Baumgarten2010}
D.~Baumgarten and J.~Haueisen.
\newblock A spatio-temporal approach for the solution of the inverse problem in
  the reconstruction of magnetic nanoparticle distributions.
\newblock {\em IEEE Transactions on Magnetics}, 46(8), August 2010.

\bibitem{Baumgarten2008}
D.~Baumgarten, M.~Liehr, F.~Wiekhorst, U.~Steinhoff, P.~Münster, P.~Miethe,
  L.~Trahms, and J.~Haueisen.
\newblock Magnetic nanoparticle imaging by means of minimum norm estimates from
  remanence measurements.
\newblock {\em Medical \& Biological Engineering \& Computing}, 46:1177--1185,
  2008.

\bibitem{benning2013ground}
M.~Benning and M.~Burger.
\newblock Ground states and singular vectors of convex variational
  regularization methods.
\newblock {\em Methods and Applications of Analysis}, 20(4):295--334, 2013.

\bibitem{Brown1963}
W.~F. Brown.
\newblock Thermal fluctuations of a single-domain particle.
\newblock {\em Physical Review}, 130:1677--1686, Jun 1963.

\bibitem{burger2013guide}
M.~Burger and S.~Osher.
\newblock A guide to the tv zoo.
\newblock In {\em Level Set and PDE Based Reconstruction Methods in Imaging},
  pages 1--70. Springer, 2013.

\bibitem{chambolle2010introduction}
A.~Chambolle, V.~Caselles, D.~Cremers, M.~Novaga, and T.~Pock.
\newblock An introduction to total variation for image analysis.
\newblock {\em Theoretical foundations and numerical methods for sparse
  recovery}, 9(263-340):227, 2010.

\bibitem{Coene2012}
A.~Coene, G.~Crevecoeur, and L.~Dupr{\'{e}}.
\newblock {Adaptive Control of Excitation Coil Arrays for Targeted Magnetic
  Nanoparticle Reconstruction Using Magnetorelaxometry}.
\newblock {\em IEEE Transactions on Magnetics}, 48(11):2842--2845, nov 2012.

\bibitem{Coene2015}
A.~Coene, G.~Crevecoeur, J.~Leliaert, L.~Dupr{\'{e}}, and G.~Crevecoeur.
\newblock {Quantitative model selection for enhanced magnetic nanoparticle
  imaging in magnetorelaxometry}.
\newblock {\em Medical Physics}, 42(12):6853, 2015.

\bibitem{Crevecoeur2012}
G.~Crevecoeur, D.~Baumgarten, U.~Steinhoff, J.~Haueisen, L.~Trahms, and
  L.~Dupré.
\newblock Advancement in magnetic nanoparticle reconstruction using sequential
  activation of excitation coil arrays using magnetorelaxometry.
\newblock {\em IEEE Transactions on Magnetics}, 48(4), April 2012.

\bibitem{egger2010forward}
H.~Egger, M.~Freiberger, and M.~Schlottbom.
\newblock On forward and inverse models in fluorescence diffuse optical
  tomography.
\newblock {\em Inverse Problems Imaging}, 4(3):411--427, 2010.

\bibitem{Engl2013}
H.~W. Engl.
\newblock {\em Integralgleichungen}.
\newblock Springer-Verlag, 2013.

\bibitem{evans1998partial}
L.~Evans.
\newblock {\em Partial Differential Equations}.
\newblock American Mathematical Society, 1998.

\bibitem{Hanson2002}
J.~D. Hanson and S.~P. Hirshman.
\newblock Compact expressions for the biot-savart fields of a filamentary
  segment.
\newblock {\em Physics of Plasmas}, 2002.

\bibitem{Hiergeist1999}
R.~Hiergeist, W.~Andrä, N.~Buske, R.~Hergt, I.~Hilger, and W.~Kaiser.
\newblock Application of magnetite ferrofluids for hyperthermia.
\newblock {\em Journal of Magnetism and Magnetic Materials}, 201(Issues
  1-3):420--422, July 1999.

\bibitem{Isakov1}
V.~Isakov.
\newblock {\em Inverse source problems}.
\newblock Number~34. American Mathematical Soc., 1990.

\bibitem{Isakov2}
V.~Isakov.
\newblock {\em Inverse problems for partial differential equations}, volume
  127.
\newblock Springer, 2006.

\bibitem{Jackson1999}
J.~D. Jackson.
\newblock {\em ClasClass Electrodynamics}.
\newblock John Wiley and Sons, Inc., 1999.

\bibitem{Liebl2012}
M.~Liebl, U.~Steinhoff, M.~Bauer, F.~Wiekhorst, L.~Trahms, D.~Baumgarten, and
  J.~Haueisen.
\newblock Spatially resolved measurement of magnetic nanoparticles using
  inhomogeneous excitation fields in the linear susceptibilty range (<1mt).
\newblock 2012.

\bibitem{Liebl2014}
M.~Liebl, U.~Steinhoff, F.~Wiekhorst, J.~Haueisen, and L.~Trahms.
\newblock Quantitative imaging of magnetic nanoparticles by magnetorelaxometry
  with multiple excitation coils.
\newblock {\em Physics in Medicine and Biology}, 2014.

\bibitem{Liebl2015}
M.~Liebl, F.~Wiekhorst, D.~Eberbeck, P.~Radon, D.~Gutkelch, D.~Baumgarten,
  U.~Steinhoff, and L.~Trahms.
\newblock {Magnetorelaxometry procedures for quantitative imaging and
  characterization of magnetic nanoparticles in biomedical applications}.
\newblock {\em Biomedical Engineering / Biomedizinische Technik},
  60(5):427--443, 2015.

\bibitem{Neel1949}
L.~Néel.
\newblock Théorie du traînage magnétique des ferromagnétiques en grains
  fins avec applications auxterres cuites.
\newblock {\em Ann. Geophys.}, 5:99--136, 1949.

\bibitem{Osher2005}
S.~Osher, M.~Burger, W.~Yin, D.~Goldfarb, and J.~Xu.
\newblock An iterative regularization method for total variation-based image
  restoration.
\newblock 4, 01 2005.

\bibitem{Sawatzky2011}
A.~Sawatzky.
\newblock {\em (Nonlocal) Total Variation in Medical Imaging}.
\newblock PhD thesis, Westfälische Wilhelms Universität Münster (WWU
  Münster), 2011.

\bibitem{Boyd2012}
B.~Wahlberg, S.~Boyd, M.~Annergren, and Y.~Wang.
\newblock An admm algorithm for a class of total variation regularized
  estimation problems.
\newblock {\em arXiv:1203.1828}, 2012.

\bibitem{Wang2004}
Z.~Wang, A.~Bovik, H.~Sheikh, and E.~Simoncelli.
\newblock Image quality assessment: From error visibility to structural
  similarity.
\newblock {\em {IEEE} Transactions on Image Processing}, 13(4):600--612, apr
  2004.

\bibitem{Wiekhorst2012}
F.~Wiekhorst, D.~Eberbeck, and L.~Trahms.
\newblock Magnetorelaxometry assisting biomedical applications of magnetic
  nanoparticles.
\newblock {\em Pharm Res}, 29:1189--1202, 2012.

\end{thebibliography}
%\normalem
%\printbibliography

\end{document}